\documentclass[journal]{IEEEtran}

\usepackage{amsmath,amsfonts,amssymb,amsmath,latexsym}

\newtheorem{theorem}{Theorem}
\newtheorem{lemma}[theorem]{Lemma}
\newtheorem{corollary}[theorem]{Corollary}
\newtheorem{proposition}[theorem]{Proposition}
\newcommand{\fd}{\mathbb{F}}
\newcommand{\z}{\mathbb{Z}}
\newcommand{\dis}{\displaystyle}

\begin{document}
%
\title{  Codes Associated with  $O^+(2n,2^r)$ \\
          and Power Moments of Kloosterman Sums}
%
%
\author{Dae San Kim,~\IEEEmembership{Member,~IEEE}
\thanks{ The author is with the Department of Mathematics, Sogang University, Seoul 121-742, Korea(e-mail; dskim@sogang.ac.kr).  }
        }
\maketitle

\begin{abstract}
In this paper, we construct three binary linear codes
$C(SO^+(2,q))$, $C(O^+(2,q))$, $C(SO^+(4,q))$, respectively
associated with the orthogonal groups $SO^+(2,q)$, $O^+(2,q)$,
$SO^+(4,q)$, with $q$ powers of two. Then we obtain recursive
formulas for the power moments of Kloosterman and 2-dimensional
Kloosterman sums in terms of the frequencies of weights in the
codes. This is done via Pless power moment identity and by
utilizing the explicit expressions of Gauss sums for the
orthogonal groups. We emphasize that, when the recursive formulas
for the power moments of Kloosterman sums are compared, the
present one is computationally more effective than the previous
one constructed from the special linear group $SL(2,q)$. We
illustrate our results with some examples.
\end{abstract}



%
\IEEEpeerreviewmaketitle

\section{Introduction}

 Let  $\psi$ be a nontrivial additive character of the finite field $\fd_q$
 with $q = p^r$ elements ( $p$ a prime), and let $m$ be a positive integer.
 Then the $m$-dimensional  Kloosterman sum $K_m(\psi;a)$(\cite{RH}) is defined by

\begin{multline*}
 K_{m}(\psi;a)=\sum_{\alpha_{1},\cdots,\alpha_{m} \in
\fd_{q}^{*}}\psi(\alpha_{1}+\cdots+\alpha_{m}+a\alpha_{1}^{-1}\cdots\alpha_{m}^{-1})\\
(a \in \fd_{q}^{*}).
\end{multline*}

In particular, if $m=1$, then $K_{1}(\psi;a)$ is simply denoted by
$K(\psi;a)$, and is called the Kloosterman sum. The Kloosterman
sum was introduced in 1926 (\cite{HDK}) to give an estimate for
the Fourier coefficients of modular forms.

For each nonnegative integer $h$, by $MK_{m}(\psi)^{h}$ we will
denote the $h$-th moment of the $m$-dimensional Kloosterman sum
$K_{m}(\psi;a)$. Namely, it is given by
\begin{equation*}
 MK_{m}(\psi)^{h}=\sum_{a \in \fd_{q}^{*}}K_{m}(\psi;a)^{h}.
 \end{equation*}

If $\psi=\lambda$ is the canonical additive character of
$\mathbb{F}_{q}$, then $MK_{m}(\lambda)^{h}$ will be simply denoted
by $MK_{m}^{h}$. If further $m=1$, for brevity $MK_{1}^{h}$ will be
indicated by $MK^{h}$. The power moments of Kloosterman sums can be
used, for example, to give an estimate for the Kloosterman sums and
have also been studied to solve a variety of problems in coding
theory over finite fields of characteristic two.

From now on, let us assume that $q=2^{r}$. Carlitz \cite{L1}
evaluated $MK^{h}$, for $h\leq 4$, while Moisio \cite{M4} computed
it for $h=6$. Recently, Moisio was able to find explicit
expressions of $MK^{h}$, for the other values of $h$ for $h \leq
10$ (cf. \cite{M1})(Similar results exist also over the finite
fields of characteristic three (cf. \cite{GR},\cite{M2})). This
was done, via Pless power moment identity, by connecting moments
of Kloosterman sums and the frequencies of weights in the binary
Zetterberg code of length $q+1$, which were known by the work of
Schoof and Vlugt in \cite{RM}.  In \cite{DS2}, the binary linear
codes $C(SL(n,q))$ associated with finite special linear groups
$SL(n,q)$ were constructed when $n,q$ are both powers of two. Then
obtained was a recursive formula for the power moments of
multi-dimensional Kloosterman sums in terms of the frequencies of
weights in $C(SL(n,q))$. This was done via Pless power moment
identity and by utilizing our previous result on the explicit
expression of the Gauss sum for $SL(n,q)$. In particular, when
$n=2$, this gives a recursive formula for the power moments of
Kloosterman sums.

In this paper, we will show the following theorem giving recursive
formulas for the power moments of Kloosterman and 2-dimensional
Kloosterman sums. To do that, we construct three binary linear
codes  $C(SO^+(2,q))$, $C(O^+(2,q))$, $C(SO^+(4,q))$, respectively
associated with $SO^+(2,q)$, $O^+(2,q)$, $SO^+(4,q)$,~and express
those power moments in terms of the frequencies of weights in each
code. Then, thanks to our previous results on the explicit
expressions of ``Gauss sums" for the orthogonal group $O^+(2n,q)$
and the special orthogonal group $SO^+(2n,q)$ \cite{DS3}, we can
express the weight of each codeword in the duals of the codes in
terms of Kloosterman or 2-dimensional Kloosterman sums. Then our
formulas will follow immediately from the Pless power moment
identity.

 The recursive formula for power moments of Kloosterman sums in this paper (cf. (1), (2))
 is computationally more effective than that in \cite{DS2}(cf. \cite{DS2}, (3)).
 This is because it is easier to compute the weight distribution of $C(SO^+(2,q))$ than that of  $C(SL(2,q))$.
 Theorem 1 in the following is the main result of this paper.

\begin{theorem}
Let $q=2^{r}$. Then we have the following.

(a) For $r \geq 3$, and $h = 1,2,\ldots,$
\begin{align}
 \begin{split}
     MK^h &= \sum_{l=0}^{h-1}(-1)^{h+l+1}{h \choose l}(q-1)^{h-l} MK^l\\
        &+q\sum_{j=0}^{min\{N_1,h\}}(-1)^{h+j}C_{1,j}\sum_{t=j}^{h}t!S(h,t)2^{h-t}{N_1-j \choose
        N_1-t},
  \end{split}
 \end{align}
where $N_1 = \mid SO^+(2,q) \mid = q-1$, and
$\{C_{1,j}\}_{j=0}^{N_1}$ is the weight distribution of
$C(SO^+(2,q))$ given by
\begin{equation}
C_{1,j} = \sum {1 \choose \nu_0} \prod_{tr(\beta^{-1}) = 0} {2
\choose \nu_\beta} (j = 0, \ldots, N_1).
\end{equation}
Here the sum is over all the sets of nonnegative integers
$\{\nu_0\} \bigcup \{\nu_\beta\}_{tr(\beta^{-1}) = 0}$ satisfying
$\nu_0 + \dis\sum_{tr(\beta^{-1}) = 0} \nu_\beta = j$ and
$\dis\sum_{tr(\beta^{-1}) = 0} \nu_\beta \beta = 0$. In addition,
$S(h,t)$ is the Stirling number of the second kind defined by
\begin{equation}
S(h,t)=\frac{1}{t!}\sum_{j=0}^{t}(-1)^{t-j}{\binom{t}{j}}j ^{h} .
\end{equation}

(b) For $r \geq 3$, and $h = 1,2,\ldots,$

 \begin{multline}
     MK^h = \sum_{l=0}^{h-1}(-1)^{h+l+1}{h \choose l}(q-1)^{h-l} MK^l\\
        +q\sum_{j=0}^{min\{N_2,h\}}(-1)^{h+j}C_{2,j}\sum_{t=j}^{h}t!S(h,t)2^{h-t}{N_2-j \choose
        N_2-t},
  \end{multline}
 where $N_2 = \mid O^+(2,q) \mid = 2(q-1)$, and
$\{C_{2,j}\}_{j=0}^{N_2}$ is the weight distribution of
$C(O^+(2,q)$ given by
\begin{equation}
C_{2,j} = \sum {q \choose \nu_0} \prod_{tr(\beta^{-1}) = 0} {2
\choose \nu_\beta} (j = 0, \ldots, N_2).
\end{equation}
Here the sum is over all the sets of nonnegative integers
$\{\nu_0\} \bigcup \{\nu_\beta\}_{tr(\beta^{-1}) = 0}$ satisfying
$\nu_0 + \dis\sum_{tr(\beta^{-1}) = 0} \nu_\beta = j$ and
$\dis\sum_{tr(\beta^{-1}) = 0} \nu_\beta \beta = 0$.\\

(c) For $r \geq 2$, and $h = 1,2,\ldots,$

 \begin{multline}
     MK_2^h = \sum_{l=0}^{h-1}(-1)^{h+l+1}{h \choose l}(q^4 - q^3 - 2q^2 + 1)^{h-l} MK_2^l\\
        +q^{1-2h}\sum_{j=0}^{min\{N_3,h\}}(-1)^{h+j}C_{3,j}\sum_{t=j}^{h}t!S(h,t)2^{h-t}{N_3-j \choose
        N_3-t},
  \end{multline}

 \begin{multline}
     MK^{2h} = \sum_{l=0}^{h-1}(-1)^{h+l+1}{h \choose l}(q^4 - q^3 - 2q^2 + q + 1)^{h-l} MK^{2l}\\
             +q^{1-2h}\sum_{j=0}^{min\{N_3,h\}}(-1)^{h+j}C_{3,j}\sum_{t=j}^{h}t!S(h,t)2^{h-t}{N_3-j \choose
        N_3-t},
  \end{multline}
 where $N_3 = \mid SO^+(4,q) \mid = q^2(q^2-1)^2$, and
$\{C_{3,j}\}_{j=0}^{N_3}$ is the weight distribution of
$C(SO^+(4,q))$ given by
\begin{equation}
C_{3,j} = \sum {m_0 \choose \nu_0} \prod_{\substack{\mid t \mid <
2\sqrt{q}\\t\equiv-1(4)}} \prod_{K(\lambda;\beta^{-1})=t} {m_t
\choose \nu_\beta} (j = 0,\ldots,N_3).
\end{equation}
Here the sum is over all the sets of nonnegative integers
$\{\nu_\beta\}_{\beta \in \fd_q}$  satisfying $\dis \sum_{\beta
\in \fd_q} \nu_\beta = j$ and $\dis\sum_{\beta \in \fd_q}
\nu_\beta \beta = 0$,
\begin{center}
$m_0 = q^3(2q^2-q-2), $
\end{center}
and
\begin{center}
 $m_t = q^2(q^3-q^2-2q+t)$,
\end{center}
for all integers $t$ satisfying $|t| < 2 \sqrt{q}$, and $t \equiv
-1 ~(mod ~4)$.
\end{theorem}

\section{$O^+(2n,q)$}

 For more details about the results of this section, one is referred to the paper \cite{DS3}.
 In addition, \cite{Wan} is an excellent reference for matrix groups over finite fields.

 Throughout this paper, the following notations will be used:\\
\begin{itemize}
 \item [] $q = 2^r$ ($r \in \z_{>0}$),\\
 \item [] $\fd_q$ = the finite field with $q$ elements,\\
 \item [] $Tr A$ = the trace of $A$ for a square matrix $A$,\\
 \item [] $^tB$ = the transpose of $B$ for any matrix $B$.
\end{itemize}\

 Let $\theta^+$ be the nondegenerate quadratic form on the vector
 space $\fd_q^{2n \times 1}$ of all $2n \times 1$ column vectors over $\fd_q$, given by
 \[
\theta^+(\sum_{i=1}^{2n} x_i e^i) = \sum_{i=1}^n x_i x_{n+i},
 \] where\\
 \{$e^1=^t[10\ldots0], e^2=^t[01\ldots0],\ldots,e^{2n}=^t[0\ldots01]$\} is the standard basis of $\fd_q^{2n \times 1}$.

 The group $O^+(2n,q)$ of all isometries of ($\fd_q^{2n \times 1}$,
 $\theta^+ $) is  given by :
 \begin{align*}
O^+(2n,q) &=
\left\{\left[%
\begin{array}{cc}
  A & B \\
  C & D \\
\end{array}%
\right] \in GL(2n,q) \Big | \substack{
^tAC, ^tBD \text{~are alternating}\\
^tAD + ^tCB = 1_n} \right\}\\
 &= \left\{\left[%
\begin{array}{cc}
  A & B \\
  C & D \\
\end{array}%
\right] \in GL(2n,q) \Big | \substack{
^tAB, ^tCD \text{~are alternating}\\
A ^tD + B ^tC = 1_n} \right\},
\end{align*}where $A,B,C,D$ are of size $n$.

An $n\times n$ matrix $A=(a_{ij})$ over $\fd_q$ is called
alternating if
\[
\left\{%
\begin{array}{ll}
    a_{ii} = 0, & \hbox{for $1 \leq i \leq n$,} \\
    a_{ij} = -a_{ji} = a_{ji} , & \hbox{for $1 \leq i < j \leq n$.} \\
\end{array}%
\right.
\]

$P^+ = P^+(2n,q)$  is the maximal parabolic subgroup of
$O^+(2n,q)$ defined by:
\[
P^+(2n,q) = \left\{ \left[%
\begin{array}{cc}
  A & 0 \\
  0 & ^tA^{-1} \\
\end{array}%
\right] \left[%
\begin{array}{cc}
  1_n & B \\
  0 & 1_n \\
\end{array}%
\right] \Big| \substack{A \in GL(n,q)\\B \text{~alternating}}
\right\}.
\]
Then, with respect to $P^+ = P^+(2n,q)$, the Bruhat decomposition
of $O^+(2n,q)$ is given by:
\begin{equation}
O^+(2n,q) = \coprod_{r=0}^n P^+ \sigma_r^+ P^+,
\end{equation}where
\[
\sigma_r^+ = \left[%
\begin{array}{cccc}
   0 & 0 & 1_r & 0 \\
  0 & 1_{n-r} & 0 & 0 \\
  1_r & 0 & 0 & 0 \\
  0 & 0 & 0 & 1_{n-r} \\
\end{array}%
\right] \in O^+(2n,q).
\]

Put, for $0\leq r \leq n$,
\[
A^+_r = \{ w \in P^+(2n,q) \mid \sigma^+_rw(\sigma_r^+)^{-1} \in
P^+(2n,q) \}.
\]
Expressing $O^+(2n,q)$ as a disjoint union of right cosets of $P^+
= P^+(2n,q)$,  the Bruhat decomposition in (9) can be written as
\begin{equation}
O^+(2n,q) = \coprod_{r=0}^n P^+ \sigma_r^+(A_r^+\backslash P^+).
\end{equation}
The order of the general linear group  $GL(n,q)$ is given by
\[
g_n = \prod_{j=0}^{n-1}(q^n - q^j) = q^{{n \choose 2}}
\prod_{j=1}^n(q^j - 1).
\]

 For integers $n,r$  with $0 \leq r \leq n$ , the $q$-binomial coefficients are defined as:
 \begin{equation}
\left[ \substack{n \\ r}
 \right]_q = \prod_{j=0}^{r-1} (q^{n-j} - 1)/(q^{r-j} - 1).
 \end{equation}

Then, for integers $n,r$  with $0 \leq r \leq n$, we have
\begin{equation}
\frac{g_n}{g_{n-r} g_r} = q^{r(n-r)}\left[ \substack{n \\ r}
 \right]_q.
\end{equation}

As it is shown in \cite{DS3},
\begin{equation}
\mid A^+_r \mid = g_r g_{n-r} q^{{n \choose 2}}q^{r(2n-3r+1)/2}.
\end{equation}

Also, it is immediate to see that
\begin{equation}
\mid P^+(2n,q) \mid = q^{{n \choose 2}} g_n.
\end{equation}

Thus we get, from (12)-(14),
\begin{equation*}
\mid A^+_r\backslash P^+(2n,q) \mid = \left[ \substack{n \\ r}
 \right]_q q^{{r \choose 2}},
\end{equation*}
and \begin{equation}
 \mid P^+(2n,q)\mid ^2 \mid A_r^+ \mid^{-1} =
q^{{n \choose 2}} g_n \left[ \substack{n \\ r}
 \right]_q q^{{r \choose 2}}.
\end{equation}
 So, from (10), (15), we get:
\begin{align}
\begin{split}
\mid O^+(2n,q) \mid &= \sum_{r=0}^n \mid P^+(2n,q)\mid^2 \mid
A_r^+ \mid ^{-1}\\
                    &= 2q^{n^2-n}(q^n - 1) \prod_{j=1}^{n-1}
                    (q^{2j} - 1),
\end{split}
\end{align}
where one can apply the following $q$-binomial theorem with
$x=-1$.
\[
\sum_{r=0}^n \left[ \substack{n \\ r}
 \right]_q (-1)^r q^{{r \choose 2}} x^r = (x;q)_n,
\]with $(x;q)_n = (1-x)(1-qx)\cdots(1-q^{n-1}x)$

( $x$ an indeterminate, $n$ a positive integer).

  There is an epimorphism of groups $\delta^+ : O^+(2n,q) \rightarrow \fd_2^+$($\fd_2^+$ denoting the additive group of $\fd_2$)
  , which is related to the Clifford algebra $C(\fd_q^{2n\times 1},\theta^+)$ of the quadratic space $(\fd_q^{2n\times 1},\theta^+)$
  , and is given by

  \[
\delta^+(w) = Tr(B~^tC),
  \]
  where
  \[
w = \left [%
\begin{array}{cc}
  A & B \\
  C & D \\
\end{array}%
\right] \in O^+(2n,q).
  \]
Then $SO^+(2n,q) := Ker ~\delta^+$ is given by
\begin{equation}
SO^+ (2n,q) = \coprod_{0 \leq r \leq n, ~r ~even} P^+
\sigma_r^+(A_r^+ \backslash P^+),
\end{equation}
and
\[
\mid SO^+(2n,q) \mid ~= q^{n^2-n}(q^n - 1) \prod_{j=1}^{n-1}
(q^{2j} -1)  (cf.~ (16)).
\]

\section{Gauss sums for $O^+(2n,q)$}

 The following notations will be used throughout this paper.

\begin{gather*}
tr(x)=x+x^2+\cdots+x^{2^{r-1}} \text{the trace function} ~\fd_q
\rightarrow \fd_2,\\
\lambda(x) = (-1)^{tr(x)} ~\text{the canonical additive character
of} ~\fd_q.
\end{gather*}
Then any nontrivial additive character $\psi$ of $\fd_q$ is given
by $\psi(x) = \lambda(ax)$ , for a unique $a \in \fd_q^*$.

For any nontrivial additive character $\psi$ of $\fd_q$ and $a \in
\fd_q^*$, the Kloosterman sum $K_{GL(t,q)}(\psi ; a)$ for
$GL(t,q)$ is defined as
\begin{equation}
K_{GL(t,q)}(\psi ; a) = \sum_{w \in GL(t,q)} \psi(Trw +
a~Trw^{-1}).
\end{equation}
Observe that, for $t=1$,~$ K_{GL(1,q)}(\psi ; a)$ denotes the
Kloosterman sum $K(\psi ; a)$.

For the Kloosterman sum $K(\psi ; a)$, we have the Weil bound (cf.
\cite{RH})
\begin{equation}
\mid K(\psi ; a) \mid \leq 2\sqrt{q}.
\end{equation}

  In \cite{DS1}, it is shown that $K_{GL(t,q)}(\psi ; a)$ ~satisfies the following recursive relation:
  for integers $t \geq 2$, ~$a \in \fd_q^*$ ,
\begin{multline}
K_{GL(t,q)}(\psi ; a) = q^{t-1}K_{GL(t-1,q)}(\psi ; a)K(\psi
;a)\\
+ q^{2t-2}(q^{t-1}-1)K_{GL(t-2,q)}(\psi ; a),
\end{multline}
where we understand that $K_{GL(0,q)}(\psi ; a)=1$ . From (20), in
\cite{DS1} an explicit expression of the Kloosterman sum for
$GL(t,q)$ was derived.\\

\begin{theorem}[\cite{DS1}] For integers $t \geq 1$, and $a \in \fd_q^*$, the
Kloosterman sum $K_{GL(t,q)}(\psi ; a)$ is given by

\begin{multline}
K_{GL(t,q)}(\psi ; a)=q^{(t-2)(t+1)/2} \sum_{l=1}^{[(t+2)/2]} q^l
K(\psi;a)^{t+2-2l}\\
\times \sum \prod_{\nu=1}^{l-1} (q^{j_\nu -2\nu}-1),
\end{multline}
where  $K(\psi;a)$ is the Kloosterman sum and the inner sum is
over all integers $j_1,\ldots,j_{l-1}$ satisfying $2l-1 \leq
j_{l-1} \leq j_{l-2} \leq \cdots \leq j_1 \leq t+1$. Here we agree
that the inner sum is $1$ for $l=1$.
\end{theorem}\

In Section 6 of \cite{DS3}, it is shown that the Gauss sums for
$O^+(2n,q)$ and $SO^+(2n,q)$ are respectively given by:

\begin{align}
\begin{split}
\sum_{w \in O^+(2n,q)} \psi(Tr w) &= \sum_{r=0}^n |A_r^+
\backslash P^+| \sum_{w \in P^+} \psi(Tr~ w\sigma_r^+)\\
 &= q^{{n \choose 2}} \sum_{r=0}^n \left[ \substack{n \\ r}
 \right]_q q^{(2rn-r^2-r)/2}s_r \\
 & \quad \times K_{GL(n-r,q)}(\psi;1),
\end{split}
\end{align}

\begin{align}
\begin{split}
\sum_{w \in SO^+(2n,q)} \psi(Tr w) &= \sum_{\substack{0\leq r \leq
n,\\r ~even}}|A_r^+
\backslash P^+| \sum_{w \in P^+} \psi(Tr ~w\sigma_r^+)\\
 &= q^{{n \choose 2}} \sum_{\substack{0\leq r \leq n,\\r ~even}} \left[ \substack{n \\ r}
 \right]_q q^{(2rn-r^2-r)/2}s_r \\
 & \quad \times K_{GL(n-r,q)}(\psi;1)
\end{split}
\end{align}
(cf. (10),(17)). Here $\psi$ is any nontrivial additive character
of $\fd_q$, $s_0 =1$, and, for $r \in \z_{> 0}$, $s_r$ denotes the
number of all $r \times r$ nonsingular symmetric matrices over
$\fd_q$ , which is given by
\begin{equation}
s_r = \left\{%
\begin{array}{ll}
    q^{r(r+2)/4}\prod_{j=1}^{r/2}(q^{2j-1}-1), & \hbox{for $r$ even,} \\
    q^{(r^2-1)/4}\prod_{j=1}^{(r+1)/2}(q^{2j-1}-1), & \hbox{for $r$ odd,} \\
\end{array}%
\right.
\end{equation}
(cf. Proposition 4.3 in \cite{DS3}).

For our purposes, we only need the following three expressions of
the Gauss sums for~$SO^+(2,q),$ $O^+(2,q)$, and $SO^+(4,q)$.  So
we state them separately as a theorem (cf. (11), (21)--(24)).
Also, for the ease of notations, we introduce
\begin{equation}
G_1(q) = SO^+(2,q), G_2(q) = O^+(2,q), G_3(q) = SO^+(4,q).
\end{equation}

\begin{theorem}
 Let $\psi$ be any nontrivial additive character of $\fd_q$. Then we have
 \begin{align*}
& \sum_{w \in G_1(q)} \psi(Tr w) = K(\psi ; 1),\\
& \sum_{w \in G_2(q)} \psi(Tr w) =  K(\psi ; 1) + q - 1,\\
& \sum_{w \in G_3(q)} \psi(Tr w) = q^2(K(\psi ; 1)^2 +q^3-q).
\end{align*}
\end{theorem}\

 For the following lemma, one notes that $(n,q-1) = 1$.\\

 \begin{lemma}
With $n=2^s (s \in \z_{\geq 0})$, the map $a \mapsto a^n : \fd_q^*
\rightarrow \fd_q^*$ is bijection.
 \end{lemma}\

 A result analogous to the following Corollary is also mentioned
 in \cite{M3}.

 \begin{corollary}
For $n=2^s(s \in \z_{\geq 0})$, and $\psi$  a nontrivial additive
character of $\fd_q$,
\[
K(\psi;a^n) = K(\psi;a).
\]
\end{corollary}
\begin{proof}
\begin{align*}
K(\psi;a^n) &= \sum_{\alpha \in \fd_q^*} \psi(\alpha + a^n
\alpha^{-1})\\
&= \sum_{\alpha \in \fd_q^*} \psi(\alpha^n + a^n
\alpha^{-n}) (\text{by Lemma 4})\\
&= \sum_{\alpha \in \fd_q^*} \psi((\alpha + a\alpha^{-1})^n)\\
&= \sum_{\alpha \in \fd_q^*} \psi(\alpha +a
\alpha^{-1})(\cite{RH},\text{Theorem 2.23(v)})\\
&= K(\psi;a).
\end{align*}
\end{proof}

 For the next corollary, we need a result of Carlitz.\\

 \begin{theorem}[\cite{L2}] For the canonical additive character $\lambda$ of $\fd_q$, and $a \in \fd_q^*$,
 \begin{equation}
K_2(\lambda;a) = K(\lambda;a)^2-q.
 \end{equation}
 \end{theorem}\

The next corollary follows from Theorems 3 and 6, Corollary 5, and
by simple change of variables.\\

\begin{corollary}
 Let  $\lambda$ be the canonical additive character of $\fd_q$, and let $a \in \fd_q^*$. Then we have
 \begin{align}
\sum_{w \in G_1(q)} \lambda(aTrw) &= K(\lambda;a),\\
\sum_{w \in G_2(q)} \lambda(aTrw) &= K(\lambda;a)+q-1,\\
\sum_{w \in G_3(q)} \lambda(aTrw) &= q^2(K(\lambda;a)^2+q^3-q)\\
                                  &= q^2(K_2(\lambda;a)+q^3).
 \end{align}
\end{corollary}\

\begin{proposition}
Let  $\lambda$ be the canonical additive character of $\fd_q$, $m
\in \z_{> 0}$, $\beta \in \fd_q$ . Then
\begin{align}
\begin{split}
& \sum_{a \in \fd_q^*} \lambda(-a \beta) K_m(\lambda;a) \\
&= \left\{%
\begin{array}{ll}
    qK_{m-1}(\lambda;\beta^{-1})+(-1)^{m+1}, & \hbox{if $\beta \neq 0$,} \\
    (-1)^{m+1}, & \hbox{if $\beta = 0$,} \\
\end{array}%
\right.
\end{split}
\end{align}
with the convention $K_0(\lambda;\beta^{-1})=\lambda(\beta^{-1})$.
\end{proposition}
\begin{proof}
(31) is equal to
\begin{align*}
& \sum_{\alpha_1,\ldots,\alpha_m \in \fd_q^*}\lambda(\alpha_1 +
\cdots + \alpha_m)\sum_{a \in \fd_q^*}
\lambda(a(\alpha_1^{-1}\cdots \alpha_m^{-1}-\beta)) \\
&= \sum_{\alpha_1,\ldots,\alpha_m \in \fd_q^*}\lambda(\alpha_1 +
\cdots + \alpha_m)\sum_{a \in \fd_q} \lambda(a(\alpha_1^{-1}\cdots
\alpha_m^{-1}-\beta))\\
 & ~~~~~~~~~~~~~~~~~~~~~~~~~~~~-\sum_{\alpha_1,\ldots,\alpha_m \in
\fd_q^*}\lambda(\alpha_1 +\cdots + \alpha_m) \\
&= q \sum \lambda(\alpha_1 + \cdots + \alpha_m) + (-1)^{m+1}.
\end{align*}
Here the sum runs over all $\alpha_1,\ldots,\alpha_m \in \fd_q^*$
satisfying $\alpha_1^{-1}\cdots \alpha_m^{-1} = \beta$, so that it
is given by
\[
\left\{%
\begin{array}{ll}
    0, & \hbox{if $\beta = 0$,} \\
    K_{m-1}(\lambda;\beta^{-1}), & \hbox{if $\beta \neq 0$,~and~$m > 1$,} \\
    \lambda(\beta^{-1}), & \hbox{if $\beta \neq 0$,~and ~$m=1$.} \\
\end{array}%
\right.
\]
\end{proof}

 Let  $G(q)$ be  one of finite classical groups over $\fd_q$. Then we put, for
 each $\beta \in \fd_q$,
 \[
N_{G(q)}(\beta) = \mid \{ w \in G(q) \mid Tr(w) = \beta \} \mid .
 \]
Then it is easy to see that
\begin{equation}
qN_{G(q)}(\beta) = \mid G(q) \mid + \sum_{a \in \fd_q^*}
\lambda(-a \beta)\sum_{w \in G(q)} \lambda(a ~Trw).
\end{equation}
For brevity, we write
\begin{equation}
n_1(\beta) = N_{G_1(q)}(\beta), n_2(\beta) = N_{G_2(q)}(\beta),
n_3(\beta) = N_{G_3(q)}(\beta).
\end{equation}
Using  (27), (28), (30)--(32), and (37), one derives the
following.

\begin{proposition}With $n_1(\beta), n_2(\beta), n_3(\beta)$ as in
(33), we have
\begin{align}
& n_1(\beta) = \left\{%
\begin{array}{ll}
    1, & \hbox{if $\beta = 0$,} \\
    2, & \hbox{if $\beta \neq 0$ with $tr(\beta^{-1}) = 0$,} \\
    0, & \hbox{if $\beta \neq 0$ with $tr(\beta^{-1}) = 1$,} \\
\end{array}%
\right. \\
& n_2(\beta) = \left\{%
\begin{array}{ll}
    q, & \hbox{if $\beta = 0$,} \\
    2, & \hbox{if $\beta \neq 0$ with $tr(\beta^{-1}) = 0$,} \\
    0, & \hbox{if $\beta \neq 0$ with $tr(\beta^{-1}) = 1$,} \\
\end{array}%
\right.\\
& n_3(\beta) = \left\{%
\begin{array}{ll}
    q^3(2q^2-q-2), & \hbox{if $\beta = 0$,} \\
    q^2\{q(q+1)(q-2)+K(\lambda;\beta^{-1})\}, & \hbox{if $\beta \neq 0$.} \\
\end{array}%
\right.
\end{align}
\end{proposition}

\section{Construction of codes}

Let
\begin{equation}
\begin{split}
& N_1=|G_1(q)|=q-1, N_2=|G_2(q)|=2(q-1),\\
& N_3=|G_3(q)|=q^2(q^2-1)^2.
\end{split}
\end{equation}
Here we will construct three binary linear codes $C(G_1(q))$ of
length $N_1$, $C(G_2(q))$ of length $N_2$ , and $C(G_3(q))$ of
length $N_3$, respectively associated with the orthogonal groups
$G_1(q)$,$G_2(q)$ , and $G_3(q)$.

By abuse of notations, for  $i=1,2,3$, let
$g_1,g_2,\ldots,g_{N_i}$ be a fixed ordering of the elements in
the group $G_i(q)$. Also, for $i=1,2,3$, we put
\[
v_i = (Trg_1,Trg_2,\ldots,Trg_{N_i}) \in \fd_q^{N_i}.
\]
Then, for  $i=1,2,3$, the binary linear code $C(G_i(q))$ is
defined as

\begin{equation}
C(G_i(q)) = \{ u \in \fd_2^{N_i} \mid u\cdot v_i = 0 \},
\end{equation}
where the dot denotes the usual inner product in $\fd_q^{N_i}$.

 The following Delsarte's theorem is well-known.\\

 \begin{theorem}[\cite{FN}]  Let $B$  be a linear code over $\fd_q$.  Then
\[
(B|_{\fd_2})^\bot = tr(B^\bot).
\]
 \end{theorem}\

 In view of this theorem, the dual $C(G_i(q))^\bot (i=1,2,3)$ is given by
\begin{equation}
C(G_i(q))^\bot = \{ c(a) = (tr(aTrg_1),\ldots,tr(aTrg_{N_i}))| a
\in \fd_q \}.
\end{equation}

Let  $\fd_2^+,\fd_q^+$ denote the additive groups of the fields
$\fd_2,\fd_q$, respectively. Then, with  $\Theta(x)=x^2+x$
denoting the Artin-Schreier operator in characteristic two, we
have the following exact sequence of groups:
\begin{equation}
0 \rightarrow \fd_2^+ \rightarrow \fd_q^+ \rightarrow
\Theta(\fd_q) \rightarrow 0.
\end{equation}
Here the first map is the inclusion and the second one is given by
$x \mapsto \Theta(x) = x^2+x$. So
\begin{equation}
\Theta(\fd_q) = \{\alpha^2 + \alpha \mid  \alpha \in \fd_q \},~
and ~~[\fd_q^+ : \Theta(\fd_q)] = 2.
\end{equation}

\begin{theorem}
 Let $\lambda$  be the canonical additive character of $\fd_q$, and let $\beta \in \fd_q^*$. Then
\begin{equation}
 (a) \sum_{\alpha \in
 \fd_q-\{0,1\}}\lambda(\frac{\beta}{\alpha^2+\alpha})=K(\lambda;\beta)-1,~~~~~~~~~~~~~~~~~~~~~~~
\end{equation}
~(b)$\dis\sum_{\alpha \in
\fd_q}\lambda(\frac{\beta}{\alpha^2+\alpha+a})=
-K(\lambda;\beta)-1$, if $x^2+x+a (a \in \fd_q)$ is irreducible
over $\fd_q$, or equivalently if $a \in
\fd_q\setminus\Theta(\fd_q)$ (cf.(41)).
\end{theorem}
\begin{proof}
(a) We compute the following sum in two different ways:
\begin{equation}
\sum_{a \in \fd_q^*} \lambda(-\beta^{-1}a)K(\lambda;a)^2.
\end{equation}
On the one hand, using (26) we see that (43) is equal to
\begin{equation}
\begin{split}
&\sum_{a \in \fd_q^*} \lambda(-\beta^{-1}a)(q+K_2(\lambda;a))\\
&=-q + \sum_{a \in \fd_q^*} \lambda(-\beta^{-1}a)K_2(\lambda;a)\\
&=-q-1+qK(\lambda;\beta) (cf. (31)).
\end{split}
\end{equation}
On the other hand, we see that (43) equals
\begin{equation*}
\begin{split}
&\sum_{\alpha_1,\alpha_2 \in
\fd_q^*}\lambda(\alpha_1+\alpha_2)\sum_{a \in
\fd_q^*}\lambda(a(\alpha_1^{-1}+\alpha_2^{-1}-\beta^{-1}))\\
&= q\sum\lambda(\alpha_1+\alpha_2)-1
\end{split}
\end{equation*}
(with the sum running over all $\alpha_1,\alpha_2 \in \fd_q^*$,
satisfying $\alpha_1^{-1}+\alpha_2^{-1}=\beta^{-1}$)
\begin{equation}
\begin{split}
~~~&=q\sum_{\alpha_1 \in
\fd_q-\{0,\beta\}}\lambda(\alpha_1+(\alpha_1^{-1}+\beta^{-1})^{-1})-1\\
&=q\sum_{\alpha_1 \in
\fd_q-\{0,\beta^{-1}\}}\lambda(\alpha_1^{-1}+(\alpha_1+\beta^{-1})^{-1})-1(\alpha_1\rightarrow \alpha_1^{-1})\\
&=q\sum_{\alpha_1 \in
\fd_q-\{0,1\}}\lambda(\frac{\beta}{\alpha_1(\alpha_1+1)})-1
(\alpha_1 \rightarrow \beta^{-1}\alpha_1).
\end{split}
\end{equation}
Equating (44) and (45), the result (a) follows.
\begin{align}
(b) & \sum_{\alpha \in \fd_q - \{0,1\}}
\lambda(\frac{\beta}{\alpha^2+\alpha}) = 2\sum_{\gamma \in
\Theta(\fd_q)-\{0\}} \lambda(\frac{\beta}{\gamma}),\\
 & \sum_{\gamma \in \Theta(\fd_q)-\{0\}}
\lambda(\frac{\beta}{\gamma}) + \sum_{\gamma \in \Theta(\fd_q)}
\lambda(\frac{\beta}{\gamma + a})
= \sum_{\gamma \in \fd_q^*}\lambda(\frac{\beta}{\gamma}) = -1.
\end{align}
So
\begin{align*}
\sum_{\alpha \in \fd_q}\lambda(\frac{\beta}{\alpha^2+\alpha+a})&=
2\sum_{\gamma \in
\Theta(\fd_q)} \lambda(\frac{\beta}{\gamma + a})(cf.(41))\\
&= -2-2\sum_{\gamma \in \Theta(\fd_q)-\{0\}}
\lambda(\frac{\beta}{\gamma})(cf.(47))\\
&= -2 -\sum_{\alpha \in \fd_q-\{0,1\}}
\lambda(\frac{\beta}{\alpha^2+\alpha})(cf.(46))\\
&=-2-(K(\lambda;\beta)-1) (cf. (42))\\
&= -1-K(\lambda;\beta).
\end{align*}
\end{proof}

\begin{theorem}
(a) For $q=2^r$, with $r \geq 3$, the map $\fd_q \rightarrow
C(G_i(q))^\bot (a \mapsto c(a))$, for $i=1,2,$ is an
$\fd_2$-linear
isomorphism. \\
(b) For any $q=2^r$, the map  $\fd_q \rightarrow
C(G_3(q))^\bot$$(a \mapsto c(a))$ is an $\fd_2$-linear
isomorphism.
\end{theorem}
\begin{proof}
(a) As $G_2(q)=O^+(2,q)$ case can be shown in exactly the same
manner, we will treat only $G_1(q)=SO^+(2,q)$ case. The map is
clearly $\fd_2$-linear and surjective. Let $a$ be in the kernel of
the map. Then $tr(aTrg)=0$, for all $g \in SO^+(2,q)$. Since
$n_1(\beta) = |\{g \in SO^+(2,q) \mid Tr(g) = \beta \}|=2$, for
all $\beta \in \fd_q^*$ with $tr(\beta^{-1})=0$(cf.
(34)),~$tr(a\beta)=0$ , for all $\beta \in \fd_q^*$ with
$tr(\beta^{-1})=0$. Hilbert's theorem 90 says that, for $\gamma
\in \fd_q$,$tr(\gamma)=0 \Leftrightarrow \gamma = \alpha^2+\alpha$
, for some $\alpha \in \fd_q$. Thus
$tr(\frac{a}{\alpha^2+\alpha})=0$, for all $\alpha \in
\fd_q\backslash\{0,1\}$. So $\dis\sum_{\alpha \in \fd_q - \{0,1\}}
\lambda(\frac{a}{\alpha^2+\alpha})=q-2$. Assume now that $a \neq
0$. Then, from (42), (19),
\[
q-2 = K(\lambda;a)-1\leq 2\sqrt{q}-1
\]
This implies that $q \leq 2\sqrt{q}+1$. But this is impossible,
since $x > 2\sqrt{x}+1$, for $x \geq 8$.\\
(b) Again, the map is $\fd_2$-linear and surjective. From (36) and
using the Weil bound in (19), it is elementary to see that
$n_3(\beta)=|\{ g \in SO^+(4,q) \mid Tr(g) = \beta)\}|>0$, for all
$\beta \in \fd_q$ . Let $a$ be in the kernel. Then $tr(aTrg)=0$,
for all $g \in SO^+(4,q)$, and hence $tr(a\beta)=0$, for all
$\beta \in \fd_q$. This implies that $a=0$, since otherwise $tr:
\fd_q\rightarrow \fd_2$ would be the trivial map.
\end{proof}\

Remark: It is easy to check that, for $i=1,2$, and $q=2^r$ with
$r=1,2$, the kernel of the map $\fd_q \rightarrow C(G_i(q))^\bot
(a \mapsto c(a))$ is $\fd_2$.

\section{Power moments of Kloosterman sums}

In this section, we will be able to find, via Pless power moment
identity, a recursive formula for the power moments of Kloosterman
sums in terms of the frequencies of weights in $C(G_i(q))$, for
each $i=1,2,3$.\\

\begin{theorem}[Pless power moment identity]
Let $B$ be an $q$-ary $[n,k]$ code, and let $B_i$ (resp.
$B_i^\bot$) denote the number of codewords of weight $i$ in
$B$(resp. in $B^\bot$).  Then, for $h=0,1,2,\ldots,$
\begin{multline}
\sum_{j=0}^n j^h B_j = \sum_{j=0}^{min\{n,h\}} (-1)^j B_j^\bot \\
\times \sum_{t=j}^h t! S(h,t) q^{k-t} (q-1)^{t-j} {n-j \choose
n-t},
\end{multline}
where $S(h,t)$ is the Stirling number of the second kind defined
in (3).
\end{theorem}\

 From now on, we will assume that $r \geq 3$, for $i=1,2,$ and hence, for $i=1,2,3$,
  every codeword in $C(G_i(q))^\bot$ can be written as $c(a)$, for a unique $a \in \fd_q$(cf. Theorem 12,
  (39)). Further, we will assume $r \geq 2$, for $i=3$,  so that
  Theorem 17 can be used in (c) of Theorem 18.\\

\begin{lemma}
Let $c(a) = (tr(aTrg_1),\cdots,tr(aTrg_{N_i})) \in
C(G_i(q))^\bot$, for $a \in \fd_q^*$, and $i=1,2,3$. Then the
Hamming weight $w(c(a))$ can be expressed as follows:
\begin{equation}
(a)~ For ~i=1,2, ~w(c(a)) =
\frac{1}{2}(q-1-K(\lambda;a)),~~~~~~~~~~~~~~~~
\end{equation}
\begin{equation}
\begin{split}
(b)~ For ~i=3, ~&w(c(a)) \\
&= \frac{1}{2}q^2(q^4-q^3-2q^2+q+1-K(\lambda;a)^2)\\
&= \frac{1}{2}q^2(q^4-q^3-2q^2+1-K_2(\lambda;a)).
\end{split}
\end{equation}
\end{lemma}
\begin{proof}
\begin{align*}
For ~i=1,2,3,~ w(c(a)) &= \frac{1}{2}
\sum_{j=1}^{N_i}(1-(-1)^{tr(aTrg_j)})\\
&= \frac{1}{2}(N_i - \sum_{w \in G_i(q)} \lambda(aTrw)).
\end{align*}
Our results now follow from (37) and (27)-(30).
\end{proof}\

 Fix $i$($i=1,2,3$), and let $u=(u_1,\ldots,u_{N_i}) \in \fd_2^{N_i}$, with $\nu_{\beta}$ 1's in the coordinate
 places where $Tr(g_j) = \beta$, for each $\beta \in \fd_q$.
 Then we see from the definition of the code $C(G_i(q))$(cf. (38)) that $u$ is a codeword with weight $j$
  if and only if  $\dis\sum_{\beta \in \fd_q} \nu_\beta = j$ and $\dis\sum_{\beta \in \fd_q} \nu_\beta \beta = 0$
  (an identity in $\fd_q$).  As there are $\dis\prod_{\beta \in \fd_q} {n_i(\beta) \choose \nu_\beta}$
   many such codewords with weight $j$, we obtain the following
   result.\\

\begin{proposition}
Let  $\{C_{i,j}\}_{j=0}^{N_i}$ be the weight distribution of
$C(G_i(q))$, for each $i=1,2,3$, where $C_{i,j}$denotes the
frequency of the codewords with weight $j$ in $C(G_i(q))$. Then
\begin{equation}
C_{i,j} = \sum \prod_{\beta \in \fd_q} {n_i(\beta) \choose
\nu_\beta},
\end{equation}
where the sum runs over all the sets of integers
$\{\nu_\beta\}_{\beta \in \fd_q}$($0 \leq \nu_\beta \leq
n_i(\beta))$, satisfying
\begin{equation}
\dis\sum_{\beta \in \fd_q} \nu_\beta = j, ~\text{and}~
\dis\sum_{\beta \in \fd_q} \nu_\beta \beta = 0
\end{equation}
\end{proposition}\

\begin{corollary}
Let $\{C_{i,j}\}_{j=0}^{N_i}$ be the weight distribution of
$C(G_i(q))$, for $i=1,2,3.$ Then, for $i=1,2,3,$ we have:
\[
C_{i,j} = C_{i,{N_i-j}},
\]for all $j$, with $0 \leq j \leq N_i$.
\end{corollary}
\begin{proof}
Under the replacements $\nu_\beta \rightarrow
n_i(\beta)-\nu_\beta$, for each $\beta \in \fd_q$, the first
equation in (52) is changed to $N_i-j$, while the second one in
(52) and the summands in (51) are left unchanged. Here the second
sum in (52) is left unchanged, since $\dis\sum_{\beta \in \fd_q}
n_i(\beta) \beta = 0$, as one can see by using the explicit
expressions of $n_i(\beta)$ in (34)--(36).
\end{proof}\

\begin{theorem}[\cite{GJ}]
Let $q=2^r$, with $r \geq 2$. Then the range $R$ of
$K(\lambda;a)$, as a varies over $\fd_q^*$, is given by:
\[
R = \{ t \in \z \mid |t|< 2\sqrt{q}, t \equiv -1 (mod ~4)\}.
\]
In addition, each value $t \in R$ is attained exactly $H(t^2-q)$
times, where $H(d)$ is the Kronecker class number of $d$.
\end{theorem}\

Now, we get the following formulas in (2), (5), and (8), by
applying the formula in (51) to each $C(G_i(q))$, using the
explicit values of $n_i(\beta)$ in (34)-(36), and taking Theorem
17 into
consideration.\\

\begin{theorem}
Let $\{C_{i,j}\}_{j=0}^{N_i}$ be the weight distribution of
$C(G_i(q))$, for $i=1,2,3$. Then
\[
(a)~~ C_{1,j} = \sum {1 \choose \nu_0} \prod_{tr(\beta^{-1})=0}{2
\choose \nu_\beta} (j=0,\ldots,N_1),~~~~~~~~~~~
\]where the sum is over all the sets of nonnegative integers  $\{\nu_0\} \cup \{\nu_\beta\}_{tr(\beta^{-1})=0}$
satisfying    $\nu_0 + \dis\sum_{tr(\beta^{-1})=0} \nu_\beta = j$
and $\dis\sum_{tr(\beta^{-1})=0} \nu_\beta \beta = 0$.
\[
(b)~~ C_{2,j} = \sum {q \choose \nu_0} \prod_{tr(\beta^{-1})=0}{2
\choose \nu_\beta} (j=0,\ldots,N_2),~~~~~~~~~~~
\]where the sum is over all the sets of nonnegative integers  $\{\nu_0\} \cup \{\nu_\beta\}_{tr(\beta^{-1})=0}$
satisfying  $\nu_0 + \dis\sum_{tr(\beta^{-1})=0} \nu_\beta = j$
and $\dis\sum_{tr(\beta^{-1})=0} \nu_\beta \beta = 0$.
\[
(c)~~ C_{3,j} = \sum {m_0 \choose \nu_0}
\prod_{\substack{|t|<2\sqrt{q}\\t\equiv-1(4)}}
\prod_{K(\lambda;\beta^{-1})=t} {m_t \choose \nu_\beta}(j =
0,\ldots,N_3),
\]where the sum is over all the sets of nonnegative integers  $\{\nu_\beta\}_{\beta \in \fd_q}$
satisfying  $ \dis\sum_{\beta \in \fd_q} \nu_\beta = j$ and
$\dis\sum_{\beta \in \fd_q} \nu_\beta \beta = 0$,
\[
m_0 = q^3(2q^2-q-2),
\]
and
\[
m_t = q^2(q^3-q^2-2q+t),
\]for all integers $t$ satisfying $|t|<2\sqrt{q}$ and $t \equiv
-1$ (mod 4).
\end{theorem}

We now apply the Pless power moment identity in (48) to each
$C(G_i(q))^\bot$, for  $i=1,2,3$, in order to obtain the results
in Theorem 1(cf. (1), (4), (6), (7)) about recursive formulas.

Then the left hand side of that identity in (48) is equal to
\begin{equation}
\sum_{a \in \fd_q^*} w(c(a))^h,
\end{equation}
with the $w(c(a))$ in each case given by (49), (50).

For $i=1,2$, (53) is
\begin{equation}
\begin{split}
& \frac{1}{2^h}\sum_{a \in \fd_q^*}(q-1-K(\lambda;a))^h\\
&=\frac{1}{2^h}\sum_{a \in \fd_q^*}\sum_{l=0}^h (-1)^l{h \choose
l} (q-1)^{h-l}K(\lambda;a)^l\\
&= \frac{1}{2^h}\sum_{l=0}^h(-1)^l {h \choose  l} (q-1)^{h-l}
MK^l.
\end{split}
\end{equation}
Similarly, for $i=3$, (53) equals
\begin{equation}
 (\frac{q^2}{2})^h \sum_{l=0}^h (-1)^l {h \choose l}
(q^4-q^3-2q^2+q+1)^{h-l} MK^{2l}
\end{equation}
\begin{equation}
 =(\frac{q^2}{2})^h\sum_{l=0}^h
(-1)^l {h \choose l} (q^4-q^3-2q^2+1)^{h-l} MK_2^l.
\end{equation}
Note here that, in view of (26), obtaining power moments of
2-dimensional Kloosterman sums is equivalent to getting even power
moments of Kloosterman sums. Also, one has to separate the term
corresponding to $l=h$ in (54)-(56), and notes
$dim_{\fd_2}C(G_i)=r$.

\section{Remarks and Examples}

  The explicit computations about power moments of Kloosterman sums was begun with the paper \cite{HS} of Sali\'{e} in 1931,
  where he showed, for any odd prime $q$,
  \begin{equation}
MK^h = q^2M_{h-1}-(q-1)^{h-1} + 2(-1)^{h-1} (h \geq1).
  \end{equation}
  However, this holds for any prime power $q=p^r$($p$ a prime). Here $M_0 = 0$, and
  for $h \in \z_{>0}$,
  \[
M_h = \mid \{ (\alpha_1,\ldots,\alpha_h) \in (\fd_q^*)^h \mid
\sum_{j=1}^h \alpha_j = 1 = \sum_{j=1}^h \alpha_j^{-1} \}\mid .
  \]

   For positive integers $h$,  we let
   \[
A_h = \mid \{ (\alpha_1,\ldots,\alpha_h) \in (\fd_q^*)^h \mid
\sum_{j=1}^h \alpha_j = 0 = \sum_{j=1}^h \alpha_j^{-1} \}\mid .
   \]
Then $(q-1)M_{h-1} = A_h$, for any $h \in \z_{>0}$.  So (57) can
be rewritten as
 \begin{equation}
MK^h = \frac{q^2}{q-1}A_h - (q-1)^{h-1} + 2(-1)^{h-1}.
\end{equation}
Iwaniec \cite{HI} showed the expression (58) for any prime $q$.
However, the proof given there works for any prime power $q$,
without any restriction. Also, this is a special case of Theorem 1
in \cite{HD}, as mentioned in Remark 2 there.

 For  $q=p$ any prime, $MK^h$ was determined for $h \leq 4$(cf. \cite{HI}, \cite{HS}).

 \begin{align*}
MK^1 &= 1, ~~~MK^2 = p^2-p-1,\\
MK^3 &= (\frac{-3}{p})p^2 + 2p + 1 \\
&(\text{with the understanding}
(\frac{-3}{2}) = -1, (\frac{-3}{3}) = 0),\\
MK^4 &= \left\{%
\begin{array}{ll}
    2p^3-3p^2-3p-1, & \hbox{$p \geq 3$;} \\
    1, & \hbox{$p=2$.} \\
\end{array}%
\right.
 \end{align*}

  Except \cite{L1} for $1 \leq h \leq 4$  and \cite{M4} for $h=6$, not much progress had been made
  until Moisio succeeded in evaluating $MK^h$, for the other values of $h$  with $h \leq 10$
   over the finite fields of characteristic two in \cite{M1} (Similar results exist also over
   the finite fields of characteristic three (cf. \cite{GR},
   \cite{M2})), So we have now closed form formulas for $h \leq 10$.

  His result was a breakthrough, but the way it was proved is too indirect,
  since the frequencies are expressed in terms of the Eichler Selberg trace formulas for the Hecke operators
  acting on certain spaces of cusp forms for $\Gamma_1(4)$.
  In addition, the power moments of Kloosterman sums are obtained only for $h \leq 10$ and not for any higher order moments.
  On the other hand, our formulas in (1) and (4) allow one, at least in principle, to compute moments of all orders for any given $q$.

   In below, for small values of $i$, we compute, by using (1), (2), and  MAGMA,
   the frequencies $C_i$ of weights in $C(SO^+(2,2^4))$ and $C(SO^+(2,2^5))$ ,
   and the power moments $MK^i$ of Kloosterman sums over $\fd_{2^4}$ and $\fd_{2^5}$.
    In particular, our results confirm those of Moisio's given in \cite{M1}, when $q=2^4$
    and $q=2^5$.\\

\begin{table}[!htp]
\begin{center}
\begin{tabular}{c c c c c c c c }
\multicolumn{8}{c}{TABLE I} \\
\multicolumn{8}{c}{The weight distribution of $C(SO^+(2,2^{4}))$} \\
\\
\hline
w & frequency & w & frequency & w & frequency & w & frequency \\[0.3pt]
\hline\\[0.5pt]
 \scriptsize{0} &\scriptsize{1}   & \scriptsize{4}  & \scriptsize{77} & \scriptsize{8} & \scriptsize{403}& \scriptsize{12}& \scriptsize{31}\\
 \scriptsize{1} &\scriptsize{1}   & \scriptsize{5}  & \scriptsize{181} & \scriptsize{9} & \scriptsize{323}& \scriptsize{13}& \scriptsize{7}\\
 \scriptsize{2} &\scriptsize{7}   & \scriptsize{6}  & \scriptsize{323} & \scriptsize{10} & \scriptsize{181}& \scriptsize{14}& \scriptsize{1}\\
 \scriptsize{3} &\scriptsize{31}   & \scriptsize{7}  & \scriptsize{403} & \scriptsize{11} & \scriptsize{77}& \scriptsize{15}& \scriptsize{1}\\
\hline
\end{tabular}
\end{center}
\end{table}

\begin{table}[!htp]
\begin{center}
\begin{tabular}{c c c c c c }
\multicolumn{6}{c}{TABLE II} \\
\multicolumn{6}{c}{The power moments of Kloosterman sums over $\fd_{2^4}$} \\
\\
\hline
i & $MK^i$ & i & $MK^i$ & i & $MK^i$ \\[0.3pt]
\hline\\[0.5pt]
\scriptsize{0} &\scriptsize{15}   & \scriptsize{10}  & \scriptsize{604249199} & \scriptsize{20} & \scriptsize{159966016268924111}\\
\scriptsize{1} &\scriptsize{1}   & \scriptsize{11}  & \scriptsize{3760049569} & \scriptsize{21} & \scriptsize{1115184421375168321}\\
\scriptsize{2} &\scriptsize{239}   & \scriptsize{12}  & \scriptsize{28661262671} & \scriptsize{22} & \scriptsize{7829178965854277039}\\
\scriptsize{3} &\scriptsize{289}   & \scriptsize{13}  & \scriptsize{188901585601} & \scriptsize{23} & \scriptsize{54689811340914235489}\\
\scriptsize{4} &\scriptsize{7631}   & \scriptsize{14}  & \scriptsize{1380879340079} & \scriptsize{24} & \scriptsize{383400882469952537231}\\
\scriptsize{5} &\scriptsize{22081}   & \scriptsize{15}  & \scriptsize{9373110103009} & \scriptsize{25} & \scriptsize{2680945149821576426881}\\
\scriptsize{6} &\scriptsize{300719}   & \scriptsize{16}  & \scriptsize{67076384888591} & \scriptsize{26} & \scriptsize{18780921149940510987119}\\
\scriptsize{7} &\scriptsize{1343329}   & \scriptsize{17}  & \scriptsize{462209786722561} & \scriptsize{27} & \scriptsize{131394922435183254906529}\\
\scriptsize{8} &\scriptsize{13118351}   & \scriptsize{18}  & \scriptsize{3272087534565359} & \scriptsize{28} & \scriptsize{920122084792925568335951}\\
\scriptsize{9} &\scriptsize{72973441}   & \scriptsize{19}  & \scriptsize{22721501074479649} & \scriptsize{29} & \scriptsize{6439066453841188580322241}\\
\hline
\end{tabular}
\end{center}
\end{table}

\begin{table}[!htp]
\begin{center}
\begin{tabular}{c c c c c c c c }
\multicolumn{8}{c}{TABLE III} \\
\multicolumn{8}{c}{The weight distribution of $C(SO^+(2,2^{5}))$} \\
\\
\hline
w & frequency & w & frequency & w & frequency & w & frequency \\[0.3pt]
\hline\\[0.3pt]
 \scriptsize{0} &\scriptsize{1}   & \scriptsize{8}  & \scriptsize{246325} & \scriptsize{16} & \scriptsize{9392163}& \scriptsize{24}& \scriptsize{81895}\\
 \scriptsize{1} &\scriptsize{1}   & \scriptsize{9}  & \scriptsize{630725} & \scriptsize{17} & \scriptsize{8285955}& \scriptsize{25}& \scriptsize{23159}\\
 \scriptsize{2} &\scriptsize{15}   & \scriptsize{10}  & \scriptsize{1385867} & \scriptsize{18} & \scriptsize{6446125}& \scriptsize{26}& \scriptsize{5369}\\
 \scriptsize{3} &\scriptsize{135}   & \scriptsize{11}  & \scriptsize{2644947} & \scriptsize{19} & \scriptsize{4410805}& \scriptsize{27}& \scriptsize{945}\\
 \scriptsize{4} &\scriptsize{945}   & \scriptsize{12}  & \scriptsize{4410805} & \scriptsize{20} & \scriptsize{2644947}& \scriptsize{28}& \scriptsize{135}\\
 \scriptsize{5} &\scriptsize{5369}   & \scriptsize{13}  & \scriptsize{6446125} & \scriptsize{21} & \scriptsize{1385867}& \scriptsize{29}& \scriptsize{15}\\
 \scriptsize{6} &\scriptsize{ 23159}   & \scriptsize{14}  & \scriptsize{8285955} & \scriptsize{22} & \scriptsize{630725}& \scriptsize{30}& \scriptsize{1}\\
 \scriptsize{7} &\scriptsize{81895}   & \scriptsize{15}  & \scriptsize{9392163} & \scriptsize{23} & \scriptsize{246325}& \scriptsize{31}& \scriptsize{1}\\
\hline
\end{tabular}
\end{center}
\end{table}

\begin{table}[!htp]
\begin{center}
\begin{tabular}{c c c c c c }
\multicolumn{6}{c}{TABLE IV} \\
\multicolumn{6}{c}{The power moments of Kloosterman sums over $\fd_{2^5}$} \\
\\
\hline
i & $MK^i$ & i & $MK^i$ & i & $MK^i$ \\[0.7pt]
\hline\\[0.7pt]
\tiny{0} &\tiny{31}   & \tiny{10}  & \tiny{44833141471} & \tiny{20} & \tiny{733937760431358760351}\\
\tiny{1} &\tiny{1}   & \tiny{11}  & \tiny{138050637121} & \tiny{21} & \tiny{6855945343839827241601}\\
\tiny{2} &\tiny{991}   & \tiny{12}  & \tiny{4621008512671} & \tiny{22} & \tiny{86346164924243497892191}\\
\tiny{3} &\tiny{-959}   & \tiny{13}  & \tiny{22291740481921} & \tiny{23} & \tiny{851252336789971927746241}\\
\tiny{4} &\tiny{63391}   & \tiny{14}  & \tiny{497555476630111} & \tiny{24} & \tiny{10249523095374924648418591}\\
\tiny{5} &\tiny{-63359}   & \tiny{15}  & \tiny{3171377872090561} & \tiny{25} & \tiny{104764273348415132423811841}\\
\tiny{6} &\tiny{5102431}   & \tiny{16}  & \tiny{55381758830599711} & \tiny{26} & \tiny{1224170008071148563308433631}\\
\tiny{7} &\tiny{-678719}   & \tiny{17}  & \tiny{423220459165032961} & \tiny{27} & \tiny{12819574031043721011365916481}\\
\tiny{8} &\tiny{460435231}   & \tiny{18}  & \tiny{6318551635327312351} & \tiny{28} & \tiny{146828974390583504114568758431}\\
\tiny{9} &\tiny{613044481}   & \tiny{19}  & \tiny{54461730980167425601} & \tiny{29} & \tiny{1562774752282717527826758007681}\\
\hline
\end{tabular}
\end{center}
\end{table}


\end{document}